\documentclass[11pt]{amsart}

\issueinfo{00}{0}{Month}{Year}
\copyrightinfo{2008}{University of Calgary}
\pagespan{1}{2}
\PII{ISSN 1715-0868}
\include{cdmstdcmds}

\usepackage{amssymb}

\usepackage[pdftex]{hyperref} 

\usepackage[capitalize]{cleveref}
\renewcommand{\Cref}{\cref}

\newcommand{\normal}{\triangleleft}
\newcommand{\normaleq}{\trianglelefteq}
\newcommand{\notnormal}{\mathrel{\not\hskip-2pt\normal}}
\newcommand{\iso}{\cong}

\DeclareMathOperator{\Cay}{Cay}

\newcommand{\dir}{\overrightarrow}
\newcommand{\Cayd}{\mathop{\dir{\Cay}}}

 \newcounter{case}
 \newenvironment{case}[1][\unskip]{\refstepcounter{case}\it
 \medskip \noindent Case \thecase\ #1. }{\unskip\upshape}
 \renewcommand{\thecase}{\arabic{case}}

\numberwithin{equation}{section}

\theoremstyle{plain}
\newtheorem{thm}[equation]{Theorem}
\newtheorem{prop}[equation]{Proposition}
\newtheorem{cor}[equation]{Corollary}
\newtheorem{lem}[equation]{Lemma}

\theoremstyle{definition}
\newtheorem*{notation}{Notation}
\newtheorem*{defn}{Definition}
\newtheorem*{term}{Terminology}

\theoremstyle{remark}
\newtheorem*{rem}{Remark}
\newtheorem{nrem}[equation]{Remark}
\newtheorem*{ack}{Acknowledgments}

\crefformat{thm}{Theorem~#2#1#3}
\crefformat{prop}{Proposition~#2#1#3}
\crefformat{cor}{Corollary~#2#1#3}
\crefformat{lem}{Lemma~#2#1#3}
\crefformat{nrem}{Remark~#2#1#3}

\makeatletter
\newcommand{\noprelistbreak}{\@nobreaktrue\nopagebreak} 
\makeatother

\begin{document}

\title[2-generated Cayley digraphs on nilpotent groups]%
{2-generated Cayley digraphs on nilpotent groups have hamiltonian paths}

\author{Dave Witte Morris}
\address{Department of Mathematics and Computer Science,
University of Lethbridge, Lethbridge, Alberta, T1K~3M4, Canada}
\email{Dave.Morris@uleth.ca, http://people.uleth.ca/$\sim$dave.morris/}

\date{(date1), and in revised form (date2).}

\begin{abstract}
Suppose $G$ is a nilpotent, finite group. We show that if $\{a,b\}$ is any $2$-element generating set of~$G$, then the corresponding Cayley digraph $\Cayd(G;a,b)$ has a hamiltonian path.
This implies that all of the connected Cayley graphs of valence $\le 4$ on~$G$ have hamiltonian paths. 
\end{abstract}

\maketitle

\section{Introduction}

Let $G$ be a group. (All groups are assumed to be finite.)

\begin{defn}
For any subset~$S$ of~$G$, the \emph{Cayley digraph of~$S$ on~$G$} is the directed graph whose vertices are the elements of~$G$, and with a directed edge $g \rightarrow gs$, for every $g \in G$ and $s \in S$. It is denoted $\Cayd(G;S)$.
\end{defn}

It is known that every connected Cayley digraph on~$G$ has a hamiltonian path if either $G$ is abelian  (see \cref{abelian}) or $G$ has prime-power order (see \cref{pn}). Since abelian groups and $p$-groups are the basic examples of nilpotent groups, it is natural to ask whether it suffices to assume that $G$ is nilpotent. We provide some evidence that this may indeed be the case:

\begin{thm} \label{2GenHP}
Every connected Cayley digraph of outvalence\/~$2$ on any nilpotent group has a hamiltonian path.
\end{thm}

There is no need to make any restriction on the outvalence of $\Cayd(G;S)$ if we assume the nilpotent group~$G$ has only one Sylow subgroup that is not abelian:

\begin{thm} \label{G=PxA}
If $G = P \times A$, where $P$ has prime-power order, and $A$ is abelian, then every connected Cayley digraph on~$G$ has a hamiltonian path.
\end{thm}

\begin{rem}
In abstract terms, the assumption $G = P \times A$ in \cref{G=PxA} is equivalent to assuming that $G$ is nilpotent and the commutator subgroup of~$G$ has prime-power order.
\end{rem}

The above results for directed graphs have the following consequence for (undirected) Cayley graphs.

\begin{cor} \label{val4}
Every connected Cayley graph of valence\/ $\le 4$ on any nilpotent group has a hamiltonian path.
\end{cor}

\begin{rem}
One can show quite easily that if $S = \{a,b\}$ is a $2$-element generating set of a group~$G$, such that $|a| = 2$, $|b| = 3$, and $|G| > 9  |ab^2|$, then $\Cayd(G; a,b)$ does not have a hamiltonian path. (This observation is attributed to J.~Milnor \cite[p.~267]{HolsztynskiStrube}.) Examples in which $G$ is (super)solvable can be constructed by taking $G$ to be an appropriate semidirect product  $\mathbb{Z}_6 \ltimes \mathbb{Z}_p$, where $p$ is a large prime that is congruent to~$1$ modulo~$6$.  Therefore,  the word ``nilpotent'' cannot be replaced with the word ``solvable'' (or even ``supersolvable'') in the statement of \cref{2GenHP}. 
\end{rem}

After some preliminaries in \cref{PrelimSect}, the above results are proved in \cref{MainPfSect}, by using the methods of \cite{Witte-pn}.
See the bibliography of \cite{M2Slovenian-LowOrder} for references on the search for hamiltonian cycles in Cayley graphs on general (non-nilpotent) groups.

\begin{ack}
This research was carried out during a visit to the University of Western Australia.  I enthusiastically thank that university, particularly the members of the School of Mathematics and Statistics, for their warm hospitality that made my visit both productive and enjoyable. 
The work was partially supported by a grant from the Natural Sciences and Engineering Research Council of Canada and by funds from Australian Research Council Federation Fellowship FF0770915.
\end{ack}

\section{Preliminaries} \label{PrelimSect}

All groups in this paper are assumed to be finite.

\begin{notation}
Let $G$ be a group, let $S$ be any subset of~$G$, and let $H$ be any subgroup of~$G$.
	\noprelistbreak
	\begin{itemize} \itemsep=\smallskipamount

	\item The \emph{Cayley digraph} $\Cayd(G;S)$ is the directed graph whose vertex set is~$G$, and with an arc from $g$ to~$gs$, for every $g \in G$ and $s \in S$.

	\item The \emph{Cayley graph} $\Cay(G;S)$ is the (undirected) graph that underlies $\Cayd(G;S)$. Thus, its vertex set is~$G$, and $g$ is adjacent to both $gs$ and~$gs^{-1}$, for every $g \in G$ and $s \in S$.

	\item $H \backslash {\Cayd(G;S)}$ denotes the digraph in which:
 		\begin{itemize}
		\item the vertices are the right cosets of~$H$,
		and
		\item there is a directed edge from $Hg$ to~$Hgs$, for each $g \in G$ and $s \in S$.
		\end{itemize}

	\item $H^G = \langle g^{-1} h g \mid h \in H, g \in G \rangle$ is the \emph{normal closure} of~$H$ in~$G$.

	\item $\langle S^{-1} S \rangle = \langle\, s_1^{-1} s_2 \mid s_1,s_2 \in S \,\rangle$ is the \emph{arc-forcing subgroup}. Note that, for any $a \in S$, we have $\langle S^{-1} S \rangle = \langle a^{-1} S \rangle = \langle\, a^{-1} s \mid s \in S \,\rangle$.

	\item For $s_1,s_2,\ldots,s_n \in S$, we use 
	$ (s_i)_{i=1}^n = (s_1,s_2,s_3,\ldots,s_n) $
to denote the walk in $\Cayd(G;S)$ that visits (in order) the vertices
	$$ e, \ s_1, \ s_1s_2, \  s_1s_2s_3, \  \ldots, \ s_1s_2 \cdots s_n .$$
Also, 
	$(s_1,s_2,s_3,\ldots,s_n)\#$
 denotes the walk  $(s_1,s_2,s_3,\ldots,s_{n-1})$ that is obtained by deleting the last term of the sequence.

	\end{itemize}
\end{notation}

\begin{term}
Contrary to most authors, we consider both $K_2$ and the loop on a single vertex to have hamiltonian cycles. This is because each of these graphs has a hamiltonian path whose terminal vertex is adjacent to its initial vertex.
\end{term}

The following well-known observation is very easy to prove.

\begin{lem}[{}{\cite[Thm.~30.3, p.~506]{Gallian-AlgText}}] \label{abelian}
Every connected Cayley digraph on any abelian group has a hamiltonian path.
\end{lem}

In the remainder of this section, we recall some useful results from \cite{Witte-pn}.

\begin{thm}[{}{Witte \cite{Witte-pn}}] \label{pn}
Every nontrivial, connected Cayley digraph on any group of prime-power order has a hamiltonian cycle.
\end{thm}

\begin{lem}[{}{cf.~\cite[Lem.~4.1]{Witte-pn}}] \label{SubnormalSeries}
Suppose $H$ is a subgroup of a group~$G$.
If $G$ is nilpotent, then there is a subnormal series 
	$$ H = H_1 \normal H_2 \normal \cdots \normal H_m = H^G $$
of~$H^G$, such that, for $1 \le k < m$:
	\begin{enumerate} \renewcommand{\theenumi}{\roman{enumi}}
	\item  \label{SubnormalSeries-conjugate}
	$H_{k+1}/H_k$ is generated by a $G$-conjugate of~$H$,
	and
	\item  \label{SubnormalSeries-order}
	$H_{k+1}/H_k$ has prime-power order.
	\end{enumerate}
\end{lem}

\begin{proof}
The desired conclusion is proved in \cite[Lem.~4.1]{Witte-pn} under the stronger assumption that $G$ is a $p$-group.
The general result follows from this special case, since every nilpotent group is a direct product of $p$-groups.

For the reader's convenience, we provide a proof from scratch: given $H_1,\ldots,H_k$, with $H_k \subsetneq H^G$, we show how to construct $H_{k+1}$. 
Since $H \subseteq H_k \subsetneq H^G$, we  have $H_k \notnormal G$, which means $N_G(H_k) \neq G$. Then, because proper subgroups of nilpotent groups are never self-normalizing \cite[Hauptsatz~III.2.3(c), p.~260]{Huppert1}, we have $N_G(H_k) \subsetneq N_G \bigl( N_G(H_k) \bigr)$, so we may choose some $x \in G$, such that 
	$$ \text{$x$ normalizes $N_G(H_k)$, but $x \notin N_G(H_k)$.} $$
Since $x \notin N_G(H_k)$, we know $x^{-1} H_k x \not\subseteq H_k$. We also know (from property~(\ref{SubnormalSeries-conjugate}) and induction) that $H_k$ is generated by $G$-conjugates of~$H$. Hence, there exists $g \in G$, such that 
	$$ \text{$g^{-1} H g \not\subseteq H_k$, and $g^{-1} H g  \subseteq x^{-1} H_k x$.} $$
Let $H_{k+1} = H_k \cdot (g^{-1} H g)$. Then:
	\begin{itemize}
	\item $H_{k+1}$ properly contains~$H$, because $g^{-1} H g \not\subseteq H_k$,
	and
	\item $H_k \normal H_{k+1}$, because $g^{-1} H g \subseteq x^{-1} H_k x \subseteq N_G(H_k)$ (since $x$ normalizes $N_G(H_k)$).
	\end{itemize}
By construction, the quotient $H_{k+1}/H_k$ is generated by $g^{-1} H g$.

Since $G$ is nilpotent, it is the direct product of its Sylow subgroups: $G = P_1 \times  \cdots \times P_r$. Hence, we may write $g = g_1  g_2  \cdots g_r$, with $g_i \in P_i$. Furthermore, every subgroup of~$G$ is the direct product of its intersections with the Sylow subgroups of~$G$. Therefore, since $g^{-1} H g \not\subseteq H_k$, there is some~$i$, such that $(g^{-1} H g) \cap P_i\not\subseteq H_k \cap P_i$. This means $g_i^{-1} H g_i \not\subseteq H_k$. We also have 
	$$g_i^{-1} H g_i 
	\subseteq \bigl\langle H, \, (g_i^{-1} H g_i) \cap P_i \bigr\rangle 
	\subseteq \langle H, g^{-1} H g \rangle 
	\subseteq \langle H_k, H_{k+1} \rangle 
	= H_{k+1} .$$
Hence, there is no harm in assuming $g = g_i \in P_i$. Then $g^{-1} H g\subseteq P_i H$, so $|H_{k+1} / H_k|$ is a divisor of~$|P_i|$, which is a prime-power.
\end{proof}

\begin{nrem} \label{SubnormalSeriesHGp}
The assumption that $G$ is nilpotent in \cref{SubnormalSeries} can be replaced with the assumption that $H^G$ is a $p$-group (for some prime~$p$). To see this, note that if $H_k \notnormal H^G$, then, since $H^G$ is nilpotent, the proof of \cref{SubnormalSeries} constructs an appropriate subgroup $H_{k+1} = H_k \cdot (g^{-1} H g)$, with $g \in H^G$. On the other hand, if $H_k \normal H^G$, then every $G$-conjugate of~$H$ normalizes $H_k$, so it is easy to construct $H_{k+1}$. (Since $|H_{k+1}/H_k|$ is a divisor of $|H^G|$, which is a power of~$p$, property~(\ref{SubnormalSeries-order})  is automatically satisfied.)
\end{nrem}

\begin{lem}[{}{\cite[Lem.~5.1]{Witte-pn}}] \label{skewed} 
Suppose $S$ generates a group~$G$, and let $H^+$ and~$H^-$ be subgroups of~$G$ with $H^- \normal H^+$. If 
	\noprelistbreak
	\begin{itemize}
	\item there is a hamiltonian cycle in $H^+ \backslash {\Cayd(G;S)}$, 
	\item every connected Cayley digraph on $H^+/H^-$ has a hamiltonian cycle,
	and
	\item $H^+/H^-$ is generated by a $G$-conjugate of the arc-forcing subgroup $\langle S^{-1} S \rangle$,
	\end{itemize}
then there is a hamiltonian cycle in $H^- \backslash {\Cayd(G;S)}$.
\end{lem}

Essentially the same proof establishes an analogous result for hamiltonian paths, but we need only the following simplified version in which $H^-$ is trivial:

\begin{lem}[{}cf.~{\cite[Lem.~5.1]{Witte-pn}}] \label{SkewedPath}
Suppose $S$ generates a group~$G$, and let $H = \langle S^{-1} S \rangle$ be the arc-forcing subgroup. If 
	\noprelistbreak
	\begin{itemize}
	\item there is a hamiltonian cycle in $H \backslash {\Cayd(G;S)}$, 
	and
	\item every connected Cayley digraph on $H$ has a hamiltonian path,
	\end{itemize}
then there is a hamiltonian path in $\Cayd(G;S)$.
\end{lem}

\begin{proof}
Let 
	\noprelistbreak
	\begin{itemize}
	\item $(s_i)_{i=1}^m$ be a hamiltonian cycle in the quotient digraph $H \backslash {\Cayd(G;S)}$,
	and
	\item $(s_1s_2\cdots s_{m-1} a_i)_{i=1}^{n-1}$ be a hamiltonian path in the Cayley digraph
$\Cayd(H; s_1s_2\cdots s_{m-1} S)$.
	\end{itemize}
Then it is not difficult to verify that
	$$(s_1, s_2, \ldots, s_{m-1}, a_i)_{i=1}^{n} \# $$
is a hamiltonian path in $\Cayd(G;S)$.
\end{proof}

\section{Proofs of the main results} \label{MainPfSect}

The heart of our argument is contained in the following result, which is adapted from the proof of \cite[Thm.~6.1]{Witte-pn}, and may be of independent interest.

\begin{prop} \label{HPinArcForcing}
Let 
	\noprelistbreak
	\begin{itemize}
	\item $S$ be a generating set of a finite group~$G$,
	and
	\item $H = \langle S^{-1} S \rangle$ be the arc-forcing subgroup.
	\end{itemize}
If 
	\noprelistbreak
	\begin{itemize}
	\item $G$ is nilpotent, 
	and
	\item every connected Cayley digraph on~$H$ has a hamiltonian path\/ {\upshape(}or hamiltonian cycle, respectively\/{\upshape)}, 
	\end{itemize}
then $\Cayd(G;S)$ has a hamiltonian path\/ {\upshape(}or hamiltonian cycle, respectively\/{\upshape)}.
\end{prop}

\begin{proof} 
Consider the subnormal series 
	$$ H = H_1 \normal H_2 \normal \cdots \normal H_m = H^G $$
that is provided by \cref{SubnormalSeries}, and choose some $a \in S$. Since 
	$$\langle a, H \rangle = \langle a,a^{-1} S \rangle = \langle S \rangle = G ,$$
we know that $a$ generates the quotient group $G/H^G  = G/H_m$. Thus, $H_m \backslash {\Cayd(G ; a)}$ is a directed cycle. Furthermore, for each~$k$, \cref{pn} tells us that every connected Cayley digraph on $H_{k+1}/H_k$ has a hamiltonian cycle. Thus, repeated application of \cref{skewed} (with $H^+ = H_{k+1}$ and $H^- = H_k$, for $k = m-1,m-2,\ldots,1$) tells us that $H_1 \backslash {\Cayd(G;S)}$ has a hamiltonian cycle. Since $H_1 = H$, this means $H \backslash {\Cayd(G;S)}$ has a hamiltonian cycle. 

By assumption, we also know that every connected Cayley digraph on~$H$ has a hamiltonian path (or hamiltonian cycle, respectively). Therefore, \cref{SkewedPath} (or \cref{skewed} with $H^+ = H$ and $H^- = \{e\}$) provides a hamiltonian path (or hamiltonian cycle) in $\Cayd(G;S)$.
\end{proof}

\begin{rem}
The proof of \cref{pn} is a minor modification of the proof of \cref{HPinArcForcing}. Namely, rather than appealing to \cref{pn} in order to know that every connected Cayley digraph on $H_{k+1}/H_k$ has a hamiltonian cycle, one can assume this is true by induction on $|G|$.
The same induction hypothesis also implies that every connected Cayley digraph on~$H$ has a hamiltonian cycle.
\end{rem}

\begin{cor} \label{ArcForcingIsAbelian}
Let  $S$ be a generating set of the group~$G$. If
	\noprelistbreak
	\begin{itemize}
	\item $G$ is nilpotent, 
	and 
	\item the arc-forcing subgroup $H = \langle S^{-1} S \rangle$ is abelian,
	\end{itemize}
then $\Cayd(G;S)$ has a hamiltonian path.
\end{cor} 

\begin{proof}
\Cref{abelian} tells us that every connected Cayley digraph on~$H$ has a hamiltonian path, so \cref{HPinArcForcing} applies.
\end{proof}

\begin{proof}[\bf Proof of \cref{2GenHP}]
Let $\{a,b\}$ be a $2$-element generating set for~$G$. Then the arc-forcing subgroup $H = \langle a^{-1} b \rangle$ is cyclic, so it is abelian. Therefore \cref{ArcForcingIsAbelian} provides a hamiltonian path in $\Cayd(G;a,b)$.
\end{proof}

\begin{proof}[\bf Proof of \cref{G=PxA}]
Let $\Cayd(G;S)$ be a connected Cayley digraph on $G = P \times A$, and let $H = \langle S^{-1} S \rangle$ be the arc-forcing subgroup.
We may assume the generating set~$S$ is minimal. 

\setcounter{case}{0}

\begin{case}
Assume $H \neq G$.
\end{case}
By induction on $|G|$, we may assume every connected Cayley digraph on~$H$ has a hamiltonian path. Then \cref{HPinArcForcing} provides a hamiltonian path in $\Cayd(G;S)$.

\begin{case}
Assume $H = G$.
\end{case}
Choose some $a \in S$, and let $\overline{\phantom{x}} \colon G \to P$ be the natural projection homomorphism. 
Since $G = H = \langle a^{-1} S - \{e\} \rangle$, and the minimal generating sets of any finite $p$-group all have the same cardinality \cite[Satz~III.3.15, p.~273]{Huppert1}, there is a proper subset $S_0$ of~$S$, such that $\langle \overline{S_0} \rangle = P$. Since $G/P \iso A$ is abelian, this implies $[G,G] \subseteq \langle S_0 \rangle$. Therefore $\langle S_0 \rangle \normaleq G$.

Let $N = \langle S_0 \rangle \normaleq G$. Since $S_0$ is a proper subset of~$S$, and $S$ is minimal, we know $N \neq G$. Also, we may assume $[G,G]$ is nontrivial, for otherwise \cref{abelian} provides a hamiltonian path in $\Cayd(G;S)$. Therefore $N$ is nontrivial. Hence, by induction on $|G|$, we may assume every connected Cayley digraph on~$N$ or $G/N$ has a hamiltonian path; let 
	\noprelistbreak
	\begin{itemize}
	\item $(s_i)_{i= 1}^n$ be a hamiltonian path in $\Cayd(N;S_0)$,
	and
	\item $(t_j)_{j=1}^q$ be a hamiltonian path in $\Cayd(G/N;S)$.
	\end{itemize}
Then it is easy to see (and well known) that
	$ \bigl( (s_i)_{i= 1}^n, t_j \bigr)_{j = 1}^{q+1} \# $
is a hamiltonian path in $\Cayd(G;S)$.
\end{proof}

\begin{proof}[\bf Proof of \cref{val4}]
Suppose $\Cay(G;S)$ is a connected Cayley graph of valence $\le 4$, and $G$ is nilpotent. 
There is no harm in assuming that the generating set~$S$ is minimal.
Let $S_2$ be the set of elements of order~$2$ in~$S$.
Also, let $P$ be the Sylow $2$-subgroup of~$G$, so $G = P \times K$, where $|K|$ is odd.

If $\#S - \#S_2 \le 1$, then, since $S_2 \subseteq P$, we know $K \iso G/P$ is cyclic. Therefore $K$ is abelian, so \cref{G=PxA} applies.

We may now assume $\#S - \#S_2 \ge 2$. Then
	\begin{align*}
	 4 &\ge \text{valence of $\Cay(G;S)$} = \#(S \cup S^{-1}) 
	 \\& = 2 (\#S - \#S_2) + \#S_2 \ge 2 (\#S - \#S_2) \ge 2 \cdot 2 
	 . \end{align*}
We must have equality throughout, so $\#S = 2$ (and $S_2 = \emptyset$). Then \cref{2GenHP} provides a hamiltonian path in $\Cay(G;S)$.
\end{proof}

The following generalization of \cref{pn} is sometimes useful.

\begin{cor}
Suppose
	\noprelistbreak
	\begin{itemize}
	 \item $S$ is a nonempty generating set of a group~$G$,
	 \item $N$ is a normal $p$-subgroup of~$G$, for some prime~$p$,
	 and
	 \item there exists $a \in G$, such that $S \subset aN$. 
	 \end{itemize}
Then $\Cayd(G;S)$ has a hamiltonian cycle.
\end{cor}

\begin{proof}
Let $H = \langle S^{-1} S \rangle \subset (aN)^{-1} (aN) = N$. Since $N \normal G$, this implies $H^G \subset N$, so $H^G$ is a $p$-group. Hence, \cref{SubnormalSeriesHGp} provides a subnormal series as in \cref{SubnormalSeries}, and \cref{pn} tells us that every connected Cayley digraph on~$H$ has a hamiltonian cycle. Then the proof of \cref{HPinArcForcing} provides a hamiltonian cycle in $\Cayd(G;S)$.
\end{proof}

\nocite{*}
\bibliographystyle{amsplain}

\end{document}